\documentclass[12pt]{article}
\usepackage{amssymb}
\usepackage[mathscr]{eucal}
\usepackage{amsthm,amsmath,anysize}
\def\a{\alpha}

\def\l{\lambda}
\def\g{\gamma}
\def\0{\bar{0}}
\def\1{\bar{1}}
\def\e{\epsilon}
\def\d{\delta}
\def\g{\mathfrak{g}}

\def\g{{\mathfrak g}}
\def\l{{\lambda}}

\newtheorem{lemma}{Lemma}[section]
\newtheorem{theorem}[lemma]{Theorem}

\newtheorem{proposition}[lemma]{Proposition}

\newtheorem{corollary}[lemma]{Corollary}

\hoffset \voffset \oddsidemargin=60pt \evensidemargin=45pt
\title{ On the simplicity of Kac modules  for the restricted Lie superalgebra $gl(m,n)$}
\author{
Chaowen Zhang\\ Department of
Mathematics,\\ China university
 of Mining and Technology,\\ Xuzhou, 221116, Jiang Su, P. R. China}
\date{ }

\begin{document}
\maketitle


\section{Introduction}  Let $\g=gl(m,n)$ be the general linear Lie superalgebra  over a field $\mathbb F$ with $\text{char.}\mathbb F=p>2$. Then $\g$  is a restricted Lie superalgebra. The Kac module of $\g$ is a finite dimensional module (see Sec.2) induced from a simple $\g_{\0}$-module. In the characteristic zero case, the simplicity of Kac modules is determined by typical weights (see \cite{k1}) which were defined using a symmetric bilinear form on $H^*$, where $H$ is the maximal torus of $\g$ consisting of diagonal matrices.  In the case that $\text{char.}\mathbb F=p>2$, the simplicity of the Kac
modules was studied in \cite{z2,zs}. In the restricted case, it was shown in \cite[Th.2.2]{zs} that the simplicity of the Kac modules is determined by a polynomial. But the conclusion that this polynomial coincides with the polynomial $P(\l)$ defined in complex number case (\cite{k1}) modulo $p$ is stated without proof(\cite[Prop. 2.1]{zs}). In more general cases, the simplicity of the Kac modules was studied in \cite{z2}, in which one has to assume $p>nm$ to determine an analogous polynomial.\par

The main goal of the present paper is to present a characteristic free approach to determine the above-mentioned polynomial. Then we study its application to the nonrestricted simple modules for $\g$. The paper is arranged as follows. Sec. 2 is the preliminaries. In Sec. 3 we discuss the simplicity of the Kac modules.  In Sec. 4 we prove the main theorem.  By applying the main theorem  we show in Sec. 5 that,  under certain conditions, the $\chi$-reduced enveloping superalgebras $u_{\chi}(\g)$ and $u_{\chi}(\g_{\0})$ are Morita equivalent.

\section{Preliminaries}
Let $\g=gl(m,n)$ be the general linear Lie superalgebra(See \cite{k1}). Then  $\g$ has  a standard basis consisting of matrices $\{e_{ij}| 1\leq i,j\leq m+n\}$.   Set  $\mathcal I=\mathcal I_0\cup\mathcal I_1$, where $$\begin{aligned} \mathcal I_0&=\{(i,j)|1\leq i<j\leq m\quad \text{or}\quad m+1\leq i<j\leq m+n\},\\\mathcal I_1&=\{(i,j)|
 1\leq i\leq m<j\leq m+n\}.\end{aligned}$$
 We denote $e_{ji}$ with $j>i$  by $f_{ij}$. Then we get $\g_{\1}=\g_{-1}\oplus\g_{1}$, where $$\g_{1}=\langle e_{ij}|(i,j)\in\mathcal I_1\rangle \quad \g_{-1}=\langle f_{ij}|(i,j)\in\mathcal I_1\rangle .$$ Let $\g^+$(resp. $\g^-$) be the subalgebra $\g_{\0}+\g_1$(resp. $\g_{\0}+\g_{-1}$) of $\g$. The parity of the basis elements is given
by $$\bar{e}_{ij}=\bar f_{ij}=\begin{cases} \bar 0,&\text{if $(i,j)\in \mathcal I_0$ or $i=j$}\\\1, &\text{if $(i,j)\in \mathcal I_1$.}\end{cases}$$

 Let $ H=\langle e_{ii}|1\leq i\leq m+n\rangle,$  and let $T$ be the linear algebraic group consisting of $(m+n)\times (m+n)$ invertible diagonal matrices. Then we have $\text{Lie} (T)=H$. Let  $\Lambda =:X(T)=\mathbb Z\e_1+\mathbb Z\e_2+\cdots +\mathbb Z\e_{m+n}.$  The set of positive roots of
   $\g$ relative to $T$ is $\Phi^+=\Phi^+_{0}\cup\Phi^+_{1},$ where $$\Phi^+_{0}=\{\e_i-\e_j|(i,j)\in \mathcal I_0\}, \quad \Phi^+_{1}=\{\e_i-\e_j|(i,j)\in \mathcal I_1\}.$$Let  $$\rho_0(m,n)=1/2\sum_{\a\in\Phi^+_0}\a, \quad \rho_1 (m,n)=1/2\sum_{\a\in\Phi^+_{1}}\a,$$  and set $\rho(m,n)=:\rho_0(m,n)-\rho_1(m,n)\in\Lambda$.\par
 Denote by $N^+$(resp. $N^-$) the Lie sub-superalgebra of $\g$ spanned by the elements $e_{ij}, (i,j)\in\mathcal I$(resp. $f_{ij},(i,j)\in\mathcal I$). By the PBW theorem (\cite{bmp}) we have the triangular decomposition of $U(\g)$: $$U(\g)\cong U(N^-)\otimes U(H)\otimes U(N^+).$$
Let $h(U(\g))$ be the set of all homogeneous elements in $U(\g)$. For each $x\in h(U(\g))$, we have a derivation $[x,-]$ on $U(\g)$ defined by $$[x,y]=xy-(-1)^{\bar x\bar y}yx, y\in h(U(\g)).$$ It is easy to see that $$[x, y_1\cdots y_t]=\sum_{i=1}^t (-1)^{\bar x\sum^{i-1}_{k=1}\bar y_k}y_1\cdots [x,y_i]\cdots y_t,$$ for $y_1,\cdots, y_t\in h(U(\g))$.

 Let $\check{\e}_i$ be the 1-psg: $G_m \longrightarrow T$ such that each $t\in G_m$ is mapped into a diagonal matrix with all entries equal to 1 but the $i$th equal to $t$ if $i\leq m$, and $t^{-1}$ if $i>m$. Then the 1-psg's $\check{\e}_i$ form a $\mathbb Z$-basis of $Y(T)$. The nondegenerate paring(\cite{hu1}): $$X(T)\times Y(T)\longrightarrow \mathbb Z: (\l, \mu)\mapsto \langle\l, \mu\rangle $$ induces a symmetric bilinear form  on $\Lambda$ defined by $$(\e_i,\e_j)=\langle \e_i, \check{\e}_j\rangle=\begin{cases} \d_{ij}, &i\leq m\\-\d_{ij}, &i>m.\end{cases}$$

 Assume the Lie superalgebra $\g$ is defined over a field $\mathbb F$. We identify $H^*$ with $\Lambda \otimes _{\mathbb Z}\mathbb F$. Then the bilinear form above is extended naturally to $H^*$. In case $\text{char.} \mathbb F=0$, this is exactly the one given in \cite{k1}.  Suppose $char.\mathbb F=p>0$.  For each $\l\in \Lambda=X(T)$, the tangent map $d\l$: $H\longrightarrow \mathbb F$, by \cite[1.2]{j1}, is a linear map satisfying $d\l(h^{[p]})=(d\l(h))^p$ for all $h\in H$. By identifying $d\check{\e}_i(1)$ with $e_{ii}$ if $i\leq m$ and $-e_{ii}$ if $i>m$, we see that $d\l$ is exactly $\l\otimes 1\in\Lambda\otimes_{\mathbb Z}\mathbb F= H^*$.
 For each $\l\in\Lambda$, we write $\l\otimes 1\in H^*$ also as $\l$.\par    Define the polynomial $f_{m,n}(\l)$ on $H^*$ by $$f_{m,n}(\l)=\Pi_{\a\in\Phi^+_1}(\l+\rho(m,n), \a), \l\in H^*.$$  $\l\in H^*$ is referred to as {\sl typical} if $f_{m,n}(\l)\neq 0$\par

\section{The simplicity of Kac modules}
Let $\g$ be the Lie superalgebra $gl(m,n)$ over a field $\mathbb F$. Let $U(\g)$(resp. $U(\g^+)$; $U(\g_{\0})$; $U(\g_{-1})$) be the universal enveloping superalgebra of $\g$(resp. $\g^+$; $\g_{\0}$; $\g_{-1}$)(see \cite{bmp}).

Let $M$ be a $U(\g_{\0})$-module. For $\mu\in H^*$, define the $\mu$-weight space of $M$ by $$M_{\mu}=\{x\in M|hx=\mu (h)x\quad\text{for all}\quad h\in H\}.$$
 For any positive odd root $\a=\e_i-\e_j\in \Phi^+_1$, $(i,j)\in\mathcal I_1$,  let $h_{\a}=[e_{ij}, f_{ij}]=e_{ii}+e_{jj}$. It is easy to check that  $\mu (h_{\a})=(\mu, h_{\a})$, for any $\mu\in \Lambda, \a\in\Phi^+_1$. Therefore we get $h_{\a}x=(\a,\mu)x$ for any $x\in M_{\mu}$. If $x\in M_{\mu}$, then $$f_{ij}x\in M_{\mu-(\e_i-\e_j)}\quad\text{ and}\quad e_{ij}x\in M_{\mu +\e_i-\e_j}$$ for any $(i,j)\in\mathcal I_0$. A nonzero vector $v^+\in M_{\mu}$ is said to be maximal if $e_{ij}v^+=0$ for all $(i,j)\in \mathcal I_0$.\par
Let $M_0(\l)$ be a simple $U(\g_{\0})$-module generated by a maximal vector of weight $\l\in H^*$. We can view $M_0(\l)$ as a $U(\g^+)$-module by letting $\g_1$ act trivially on it. Then the induced $U(\g)$-module  $$K(\l)=U(\g)\otimes_{U(\g^+)}M_0(\l)$$ is called a Kac module. In case $\mathbb F=\mathbb C$, \cite[Prop. 2.9]{k1} says that $K(\l)$ is simple  if and only if $\l$ is typical.\par
 Note: In the case that $\text{char.}\mathbb F>0$, the maximal vector in the $U(\g_{\0})$-module $M_0(\l)$ and hence its weight $\l\in H^*$ may not be unique.\par

 Let $M_0(\l)$ be a simple $U(\g_{\0})$-module generated by a maximal vector of  weight $\l\in H^*$. If $\text{char.}\mathbb F= p>0$, then   $M_0(\l)$ is finite dimensional by the Jacobson's theorem(see \cite[Th. 2.4]{sf}), so is  the   Kac module $K(\l)$.\par
By definition, we have $$K(\l)\cong U(\g_{-1})\otimes_{\mathbb F}M_0(\l)$$ as $U(\g_{-1})$-modules.\par  Let us define an order on the set $\mathcal I_1$ as follows:  $$(i,j)\prec (s,t)\quad \text{if and only if}\quad j>t\quad \text{or}\quad j=t\quad \text{but}\quad i<s.$$ We write $(i,j)\preceq (s,t)$ if $(i,j)\prec (s,t)$ or $(i,j)=(s,t)$. Define $f_{ij}\prec f_{st}$ if $(i,j)\prec (s,t)$.  For each subset $I\subseteq \mathcal I_1$, let $f_I$ denote the product $\Pi_{(i,j)\in I} f_{ij}$ in this order. In particular, $f_{\phi}=1$. Then it is clear that $K(\l)$ has a basis  $$f_I \otimes v_i, I\subseteq \mathcal I_1, i=1,\cdots s,$$  where $v_1,\cdots, v_s$ is a basis of $M_0(\l)$. Note that $f_{ij}f_{st}=-f_{st}f_{ij}\in U(\g_{-1})\subseteq U(\g)$, for any $(i,j), (s,t)\in\mathcal I_1$.   Then it is easily seen that, by multiplying appropriate $f_{ij}'s$( $(i,j)\in \mathcal I_1$) to any nonzero element $x\in K(\l)$,  one obtains  $f_{\mathcal I_1}\otimes v$ with $0\neq v\in M_0(\l)$.\par
 It is easy to check that, for any $(i,j)\in \mathcal I_0$,  $$ (*)\quad f_{ij}f_{\mathcal I_1}=f_{\mathcal I_1}f_{ij}, \quad e_{ij}f_{\mathcal I_1}=f_{\mathcal I_1}e_{ij}.$$
For each subset $I\subseteq \mathcal I_1$, we denote by $e_I$ the product $\Pi_{(i,j)\in I}e_{ij}$ in the reversed order of $\mathcal I_1$. For each $(i,j)\in\mathcal I_1$, let $$>(i,j)(\text{resp. $\geq (i,j); <(i,j); \leq (i,j)$; })$$ denote the subset $$\begin{aligned} &\{(s,t)\in\mathcal I_1|(s,t)\succ (i,j)\}\\&(\text{resp.} \{(s,t)\in\mathcal I_1|(s,t)\succeq (i,j)\}; \\&\{(s,t)\in\mathcal I_1|(s,t)\prec (i,j)\};\\& \{(s,t)\in\mathcal I_1|(s,t)\preceq (i,j)\}).\end{aligned}$$ For $(i,j), (s,t)\in\mathcal I_1$ with $(i,j)\prec (s,t)$, we denote by $((i,j), (s,t))$ the subset $\{(i',j')\in\mathcal I_1|(i,j)\prec (i',j')\prec (s,t)\}$.\par

 Let us write $e_{\mathcal I_1}f_{\mathcal I_1}\in U(\g)$ in terms of the triangular decomposition of $U(\g)$(see Sec.2):  $$e_{\mathcal I_1}f_{\mathcal I_1}=f(h)+\sum u^-_iu^0_iu^+_i, u_i^{\pm}\in U(N^{\pm}), f(h), u^0_i\in U(H).$$ Note that $U(\g)$ is a $T$-module under the adjoint action. Denote by $\text{wt} (u)$ the weight of a weight vector $u\in U(\g)$. Then since $\text{wt}(e_{\mathcal I_1}f_{\mathcal I_1})=0$,
 so that $\text{wt}(u_i^+)=-\text{wt}(u^-_i)$
 we get $u_i^+\in\mathbb F$ if and only if $u^-_i\in\mathbb F$. \par Let $v_{\l}$ be a maximal vector in $M_0(\l)\subseteq K(\l)$. Then we get $$e_{\mathcal I_1}f_{\mathcal I_1}v_{\l}=f(h)v_{\l}=f(h)(\l)v_{\l}.$$
 The following proposition was proved in \cite{k2} in the case $\mathbb F=\mathbb C$, and proved in \cite{zs} in the restricted case.
 \begin{proposition} Let $\l\in H^*$. Then $ K(\l)$ is simple if and only if $ f(h)(\l)\neq 0$.
\end{proposition}\begin{proof} By the formula $(*)$ above, the subspace $f_{\mathcal I_1}
\otimes M_0(\l)\subseteq K(\l)$ is also a simple $U(\g_{\0})$-module, and which is clearly  annihilated by $\g_{-1}$. It follows that $$U(\g_1)f_{\mathcal I_1}\otimes M_0(\l)=U(\g)f_{\mathcal I_1}
\otimes M_0(\l),$$  is a $U(\g)$-submodule of $ K(\l)$. \par Suppose  $K(\l)$ is simple. Then  $U(\g_1)f_{\mathcal I_1}
\otimes M_0(\l)= K(\l).$ Since  $\text{dim}U(\g_1)=\text{dim}U(\g_{-1})$, $K(\l)$ has  a basis consisting of elements $$e_Im_s, \quad I\subseteq \mathcal I_1, s=1,\dots, r,$$  where $m_1,\dots,m_r$ is a basis of $f_{\mathcal I_1}\otimes M_0(\l)$. Choose $m_1=f_{\mathcal I_1}\otimes v_{\l}$, where $v_{\l}\in M_0(\l)$ is a maximal vector of weight $\l$.  Then we have $$0\neq e_{\mathcal I_1}f_{\mathcal I_1}\otimes v_{\l}=1\otimes f(h)(\l)v_{\l},$$ and hence $ f(h)(\l)\neq 0$.\par
 Suppose  $ f(h)(\l)\neq 0$. Assume $K=K_{\0}\oplus K_{\1}$ is a nonzero submodule of $ K(\l)$. Let $x\in h(K)$ be a nonzero vector.  Apply appropriate $f_{ij}$'s to $x$ to obtain $f_{\mathcal I_1}\otimes m\in K$
 with $0\neq m\in M_0(\l)$. From the formula $(*)$ above it follows that $f_{\mathcal I_1}\otimes v_{\l}\in K$,  since $M_0(\l)$ is a simple $U(\g_{\0})$-module. Then we get $$e_{\mathcal I_1}f_{\mathcal I_1}\otimes v_{\l}=1\otimes f(h)(\l)v_{\l}\in K,$$ so that $v_{\l}\in K$. This gives $K= K(\l)$, and hence $K(\l)$ is simple.
\end{proof}

   \section{The main theorem}
   In this section we give a characteristic free approach to determine the polynomial $f(h)(\l)$.

\begin{lemma} Let $1\leq i\leq m$. Then $e_{i,m+n}f_{>(i,m+n)}v_{\l}=0.$
\end{lemma}
\begin{proof} For each $(s,t)\succ (i,m+n)$, we have $$[e_{i,m+n}, f_{st}]=\begin{cases}e_{t,m+n},&\text{if $i=s, t<m+n$}\\e_{is},&\text{if $i<s, t=m+n$}\\0,&\text{otherwise.}\end{cases}$$Then we have  $$\begin{aligned}
 e_{i,m+n}f_{> (i,m+n)}v_{\l}&=[e_{(i,m+n)}, f_{> (i,m+n)}]v_{\l}\\&=\sum_{f_{st}\succ f_{i,m+n}}(-1)^{\a_{st}}f_{((i,m+n),(s,t))} [e_{i,m+n}, f_{st}]f_{> (s,t)}v_{\l}\\&=\sum_{s>i, t=m+n}(-1)^{\a_{st}} f_{((i,m+n),(s,t))} e_{is}f_{>(s,m+n)}v_{\l}\\&+\sum_{s=i,t<m+n}(-1)^{\a_{st}} f_{((i,m+n),(s,t))} e_{t,m+n}f_{>(s,t)}v_{\l},\end{aligned}$$ where $\a_{st}\in \mathbb Z_2$. Note that the second summation is equal to zero, since $e_{t,m+n}$ commutes with all $f_{ij}$($(i,j)\in\mathcal I_1$) with $f_{ij}\succ f_{st}$. \par
  We claim that the first summation is also equal to zero. In fact, we have, in the case where $s>i, t=m+n$,   $$\begin{aligned}
 e_{is}f_{>(s,m+n)}v_{\l}& =[e_{is}, f_{>(s,m+n)}]v_{\l}\\&=\sum ^{m+1}_{j=m+n-1}f_{((s,m+n),(i,j))} [e_{is}, f_{ij}]f_{>(i,j)}v_{\l}\\&=\sum ^{m+1}_{j=m+n-1}f_{((s,m+n),(i,j))}f_{sj}f_{>(i,j)}v_{\l}=0,\end{aligned}$$ where the last equality is given by the fact that $f_{sj}\succ f_{ij}$.
\end{proof}

  \begin{theorem}For each $\l\in H^*$, we have $f(h)(\l)=f_{m,n}(\l).$
 \end{theorem}\begin{proof}We have $$\begin{aligned}
e_{\mathcal I_1}f_{\mathcal I_1}v_{\l}&=e_{>(1,m+n)}(e_{1,m+n}f_{1,m+n})f_{>(1,m+n)}v_{\l}\\&=e_{> (1,m+n)}(e_{11}+e_{m+n,m+n})f_{>(1,m+n)}v_{\l}\\&-e_{> (1,m+n)}f_{1,m+n}e_{1,m+n}f_{> (1,m+n)}v_{\l}\\
 (\text{Using Lemma 4.1})&=e_{>(1,m+n)}(e_{11}+e_{m+n,m+n})f_{> (1,m+n)}v_{\l}\\&=(\l+\a_1)(e_{11}+e_{m+n,m+n})e_{>(1,m+n)}f_{>(1,m+n)}v_{\l}\\&=(\l+\a_1, \e_1-\e_{m+n})e_{>(1,m+n)}f_{>(1,m+n)}v_{\l},\end{aligned}$$ where $\l+\a_1$ is the weight of $f_{> (1,m+n)}v_{\l}$.\par  Using Lemma 4.1, we compute $e_{>(1,m+n)}f_{>(1,m+n)}v_{\l}$ in a similar way.  Continue the process,   we  get $$e_{\mathcal I_1}f_{\mathcal I_1}v_{\l}=\Pi^k_{i=1}(\l+\a_i)(e_{ii}+e_{m+n,m+n}) e_{>(k,m+n)}f_{> (k,m+n)}v_{\l}$$$$=\cdots$$$$=\Pi^m_{i=1}(\l+\a_i)(e_{ii}+e_{m+n,m+n}) e_{\geq(1,m+n-1)}f_{\geq (1,m+n-1)}v_{\l}$$$$=\Pi^m_{i=1}(\l+\a_i, \e_i-\e_{m+n}) e_{\geq(1,m+n-1)}f_{\geq (1,m+n-1)}v_{\l}.$$  For each $1\leq i\leq m$,  it is easily seen that the weight of $f_{>(i,m+n)}v_{\l}$ is  $$\l+\a_i=\l-2\rho_1(m,n)+\sum^i_{k=1}(\e_k-\e_{m+n}).$$
 By a straightforward computation we have,
   for each  $1\leq i\leq m$, $$(1)\quad (-\rho_0(m,n)-\rho_1(m,n)+\sum_{k=1}^i(\e_k-\e_{m+n}),\e_i-\e_{m+n})=0.$$
 Applying the formula (1), we have  $$(\l+\a_i,\e_i-\e_{m+n})=(\l+\rho(m,n),\e_i-\e_{m+n})$$ for any $1\leq i\leq m$, which gives $$f(h)(\l)=\Pi^m_{k=1}(\l+\rho(m,n),\e_k-\e_{m+n})e_{\geq (1,m+n-1)}f_{\geq (1,m+n-1)}v_{\l}.$$
 We now prove the theorem by  induction on $n$. The case $n=1$ follows immediately from the above equation. Assume the case $n-1$ and consider the case $n$.
 Note that  $$\rho(m,n)=\rho (m,n-1)+\frac{1}{2}[\sum_{k>m}(\e_k-\e_{m+n})-\sum_{k\leq m}(\e_k-\e_{m+n})].$$ By a short computation we have $$(\sum_{k>m}(\e_k-\e_{m+n})-\sum_{k\leq m}(\e_k-\e_{m+n}),\e_i-\e_j)=0$$ for all $i\leq m<j<m+n$, so that $$(\l+\rho (m,n-1),\e_i-\e_j)=(\l+\rho(m,n),\e_i-\e_j).$$ Then we have by the induction hypothesis that  $$f(h)(\l)=\Pi^m_{k=1}(\l+\rho(m,n),\e_k-\e_{m+n}) \Pi_{i\leq m<j<m+n}(\l+\rho (m,n-1),\e_i-\e_j)$$$$=\Pi_{(i,j)\in\mathcal I_1}(\l+\rho(m,n),\e_i-\e_j)=f_{m,n}(\l).$$

\end{proof}
\begin{corollary} Let $\g=gl(m,n)$ be defined over $\mathbb F$ with $\text{char.}\mathbb F=p>2$. Then  $K(\l)$ is simple if and only if $\Pi_{(i,j)\in\mathcal I_1}(\l+\rho(m,n), \e_i-\e_j)\neq 0$.
\end{corollary}
\section{Applications of the main theorem}
In this section, we assume $\g=gl(m,n)$ is defined over a field $\mathbb F$ of characteristic $p>0$. We abbreviate $\rho(m,n)$ to $\rho$.  Then $\g$ is a restricted Lie superalgebra (see\cite{bmp}) with the $p$-map $[p]$ the $p$th power map. \par By \cite{wz,z3}, each simple $\g$-module $M=M_{\0}\oplus M_{\1}$ affords a $p$-character $\chi\in \g_{\0}^*$ such that $$(x^p-x^{[p]}-\chi(x)^p)v=0$$ for all $x\in \g_{\0}, v\in M$. Let $u_{\chi}(\g)$(resp. $u_{\chi}(\g_{\0})$; $u_{\chi}(\g^+)$; $u_{\chi}(\g^-)$) be the quotient superalgebra of $U(\g)$(resp. $U(\g_{\0})$; $U(\g^+)$; $U(\g^-)$) by its $\mathbb Z_2$-graded two-sided ideal generated by the central elements $$x^p-x^{[p]}-\chi(x)^p, x\in \g_{\0}.$$ Then $M$ is a $u_{\chi}(\g)$-module. The superalgebras $u_{\chi}(\g_{\0})$, $u_{\chi}(\g^+)$ and $u_{\chi}(\g^-)$ are all viewed canonically as subalgebras of $u_{\chi}(\g)$. \par
 Recall the subalgebras $\g_1$ and $\g_{-1}$ of $\g$. Let $U(\g_1)$ and $U(\g_{-1})$ be their enveloping algebras, and $u(\g_1)$ and $u(\g_{-1})$ their images in $u_{\chi}(\g)$.  \begin{lemma} There is a $\mathbb F$-vector space isomorphism: $$u_{\chi}(\g)\cong U(\g_{-1})\otimes u_{\chi}(\g_{\0})\otimes U(\g_1).$$
 \end{lemma}\begin{proof}Let $I_{\chi}$(resp. $I^0_{\chi}$) be the two-sided ideals of $U(\g)$(resp. $U(\g_{\0})$) generated by the central elements(cf \cite{sf}) $x^p-x^{[p]}-\chi(x)^p$, $x\in \g_{\0}$. By the PBW theorem(see \cite{bmp}) we get $$U(\g)\cong U(\g_{-1})\otimes U(\g_{\0})\otimes U(\g_1).$$ Note that the central elements $x^p-x^{[p]}-\chi(x)^p$, $x\in \g_{\0}$ are all contained in $U(\g_{\0})$. Then we have  $$\begin{aligned}
 I_{\chi}&=\sum _x U(\g)(x^p-x^{[p]}-\chi(x)^p)\\&=\sum_x U(\g_{-1})U(\g_{\0}) U(\g_1)(x^p-x^{[p]}-\chi(x)^p)\\&=U(\g_{-1})\sum_xU(\g_{\0})(x^p-x^{[p]}-\chi(x)^p)U(\g_1)\\&= U(\g_{-1}) I^0_{\chi}U(\g_1)\\&\cong U(\g_{-1})\otimes I_{\chi}^0\otimes U(\g_1),\end{aligned}$$ which gives $$\begin{aligned}
 u_{\chi}(\g)&\cong U(\g_{-1})\otimes U(\g_{\0})\otimes U(\g_1)/U(\g_{-1})\otimes I^0_{\chi}\otimes U(\g_1)\\&\cong U(\g_{-1})\otimes u_{\chi}(\g_{\0})\otimes U(\g_1).\end{aligned}$$
 \end{proof}
 It follows from the lemma that $u(\g_{\pm 1})\cong U(\g_{\pm 1}).$\par
  To study the representations of $u_{\chi}(\g)$,   by applying the automorphisms of the Lie superalgebra $\g$ we may assume $\chi(e_{ij})=0$ for all $(i,j)\in\mathcal I_0$(see \cite{wz}).
Let $M$ be a simple $u_{\chi}(\g_{\0})$-module. Then $M$ contains a maximal vector $v^+$ of weight $\mu$ for some $\mu \in H^*$(cf. \cite{fp}). Since $M$ is simple, $M$ is generated by $v^+$. Denote $M$ by $M(\mu)$. We view $M(\mu)$  as a $u_{\chi}(\g^+)$-module annihilated by $\g_1$. \par Set $$K_{\chi}(\mu)=u_{\chi}(\g)\otimes _{u_{\chi}(\g^+)} M(\mu).$$ Let us note that the maximal vector and the weight $\mu$ need not be unique.\par
 Let $\pi: U(\g)\longrightarrow u_{\chi}(\g)$ be the canonical epimorphism. Then $\pi$ maps $U(\g^+)$(resp. $U(\g^-)$; $U(\g_{\0})$) onto $u_{\chi}(\g^+)$(resp. $u_{\chi}(\g^-)$; $u_{\chi}(\g_{\0})$). Using the epimorphism  $M(\l)$ can be viewed as a $U(\g^+)$-module annihilated by $\g_1$. Define the $\mathbb F$-linear mapping $$\phi: K(\l)=U(\g)\otimes_{U(\g^+)}M(\l)\longrightarrow K_{\chi}(\l)$$ such that  $\phi(u\otimes x)=\pi (u)\otimes x$ for all $u\in U(\g)$, $x\in M(\l)$. It is easily seen that $\phi$ is a $U(\g)$-module epimorphism. By comparing dimensions we obtain that $\phi$ is an isomorphism. \par For  $e_{\mathcal I_1}$, $f_{\mathcal I_1}\in U(\g)$, let us denote their images  in  $u_{\chi}(\g)$ also by the same notation. Let $v_{\l}$ be  a maximal vector in the $u_{\chi}(\g^+)$-module $M(\l)$ of weight $\l$. Then by applying the isomorphism $\phi$ we see that in $K_{\chi}(\l)$ $$\begin{aligned} e_{\mathcal I_1}f_{\mathcal I_1}\otimes v_{\l}&=\phi (e_{\mathcal I_1}f_{\mathcal I_1}\otimes v_{\l})\\&=\phi(\Pi_{\a\in\Phi^+_1}(\l+\rho, \a)v_{\l})\\&=\Pi_{\a\in\Phi^+_1}(\l+\rho)(h_{\a})v_{\l}.\end{aligned}$$

For any $u_{\chi}(\g_{\0})$-module $M$, $M$ is viewed as a $u_{\chi}(\g^+)$-module by letting $\g_1M=0$,  we define the induced functor from the categories of $u_{\chi}(\g_{\0})$-modules to the categories of $u_{\chi}(\g)$-modules by $$\text{Ind}(M)= u_{\chi}(\g)\otimes_{u_{\chi}(\g^+)} M.$$ Clearly Ind is an exact functor and $\text{Ind}(M)=K_{\chi}(\l)$ in case $M$ is a simple $u_{\chi}(\g_{\0})$-module $M(\l)$.  \par
 For any $u(\g_1)$-module $N=N_{\0}\oplus N_{\1}$, let us denote $$N^{\g_1}=\{x\in N|gx=0 \quad\text{for any}\quad g\in \g_1\}.$$
  \begin{lemma}For the left regular $u(\g_1)$-module $u(\g_1)$, we have $u(\g_1)^{\g_1}=\mathbb F e_{\mathcal I_1}$. \end{lemma}
  \begin{proof} Clearly we have $e_{\mathcal I_1}\in u(\g_1)^{\g_1}$. Since $u(\g_1)\cong U(\g_1)$,  $u(\g_1)$ has a basis $e_I, I\subseteq \mathcal I_1$. For any nonzero $x=\sum c_Ie_I\in u(\g_1)$, if there is $I\subsetneq \mathcal I_1$ such that $c_I\neq 0$, then we have $e_{ij}x\neq 0$ for some $(i,j)\in\mathcal I_1$. This gives $u(\g_1)^{\g_1}=\mathbb Fe_{\mathcal I_1}$
  \end{proof}
\begin{theorem}If $\chi (h_{\a})\neq 0$ for all $\a\in \Phi_1^+$, then $u_{\chi}(\g_{\0})$ and $u_{\chi}(\g)$ are Morita equivalent.
\end{theorem} First we show that $K_{\chi}(\l)$ is simple for any $\l\in H^*$. Let $N=N_{\0}\oplus N_{\1}$ be a nonzero submodule of $K_{\chi}(\l)$. Take a nonzero elements $v\in N$, by applying appropriate $f_{ij}$($(i,j)\in\mathcal I_1$) we get $f_{\mathcal I_1}\otimes x\in N$ for some $0\neq x\in M(\l)$. We may assume $x$ is a weight vector.i.e., $x\in M(\l)_{\mu}$ for some $\mu\in H^*$. Since $M(\l)$ is a simple $u_{\chi}(\g_{\0})$-module, we have $u_{\chi}(\g_{\0})x=M(\l)$. Hence,  there is an element $$f=\sum c_iu_i^-u^0_iu_i^+\in u_{\chi}(\g_{\0})$$ such that $fx=v_{\l}$, where $u_i^-$(resp. $u^+_i$; $u^0_i$) is the product of $f_{ij}$(resp. $e_{ij}; e_{ss}$), $(i,j)\in \mathcal I_0$, $1\leq s\leq m+n$, $c_i\in \mathbb F$ and $v_{\l}$ is a maximal vector  of the weight $\l$. \par Since $x$ is a weight vector, we may assume $f=\sum c_iu_i^-u_i^+$. Since each $e_{ij}$ and $f_{ij}$ with $(i,j)\in\mathcal I_0$ commutes with $f_{\mathcal I_1}$, by applying $f$ to $f_{\mathcal I_1}\otimes x\in N$ we get $f_{\mathcal I_1}\otimes v_{\l}\in N$. Applying $e_{\mathcal I_1}$ to
which we get $$\Pi_{(i,j)\in\mathcal I_1}(\l+\rho,\e_i-\e_j) v_{\l}=\Pi_{\a\in\Phi^+_1}(\l +\rho)(h_{\a})v_{\l}\in N.$$ Note that $h_{\a}^p-h_{\a}^{[p]}-\chi(h_{\a})^p$ in $u_{\chi}(\g)$, so that $$\l (h_{\a})^p-\l (h_{\a})=\chi(h_{\a})^p,$$ which implies that $\l(h_{\a})\notin \mathbb F_p$. Since $\rho (h_{\a})=(\rho, \a)\in \mathbb F_p$ for any $\a\in \Phi^+_1$, we get $\Pi_{\a\in\Phi^+_1}(\l +\rho)(h_{\a})\neq 0$, which gives $v_{\l}\in N$. Therefore $N=K_{\chi}(\l)$, so that $K_{\chi}(\l)$ is simple.\par
 Next we show that
 $K_{\chi}(\l)^{\g_1}=M(\l)$. Note that the subspace $f_{\mathcal I_1}\otimes M(\l)\subseteq K(\l)$ is annihilated by $\g_{-1}$. Since $e_{ij}, f_{ij}, (i,j)\in\mathcal I_0$ commutes with $f_{\mathcal I_1}$, the subspace is a simple $u_{\chi}(\g^-)$-submodule of $K_{\chi}(\l)$.
 Since $K_{\chi}(\l)$ is simple, we have $$\begin{aligned} K_{\chi}(\l)&=u_{\chi}(\g)f_{\mathcal I_1}\otimes M(\l)\\ &=u(\g_1)u_{\chi}(\g_{\0})u(\g_{-1})f_{\mathcal I_1}\otimes M(\l)\\& =u(\g_1)f_{\mathcal I_1}\otimes M(\l).\end{aligned}$$ Set $$K^-_{\chi}(f_{\mathcal I_1}\otimes M(\l))=u_{\chi}(\g)\otimes _{u_{\chi}(\g^-)}(f_{\mathcal I_1}\otimes M(\l)),$$ where $f_{\mathcal I_1}\otimes M(\l)$ is viewed as a $u_{\chi}(\g^-)$-module annihilated by $\g_{-1}$. By the comparison of dimensions we have that $K_{\chi}(\l)$ is isomorphic to $K^-_{\chi}(f_{\mathcal I_1}\otimes M(\l))$ as $u_{\chi}(\g)$-modules. Thus, as $u(\g_1)$-modules, we have $$K_{\chi}(\l)\cong u(\g_1)\otimes_{\mathbb F} f_{\mathcal I_1}\otimes M(\l),$$ from which it follows that $$\begin{aligned} K_{\chi}(\l)^{\g_1}&\cong u(\g_1)^{\g_1}\otimes f_{\mathcal I_1}\otimes M(\l)\\&\cong e_{\mathcal I_1}f_{\mathcal I_1}\otimes M(\l)\\&=M(\l),\end{aligned}$$ where the last equality is given by the fact that $e_{\mathcal I_1}f_{\mathcal I_1}v_{\l}\neq 0$.\par
 From above discussion, we have that the functor $(,)^{\g_1}$ is right adjoint to Ind. By a similar argument as that for \cite[Th. 3.2]{fp}, $u_{\chi}(\g_{\0})$ and $u_{\chi}(\g)$ are Morita equivalent.

\def\refname{
\end{document}